\documentclass[12pt,reqno,a4paper]{amsart}

\usepackage{amsfonts}
\usepackage{amsmath}
\usepackage{amssymb}
\usepackage{amstext}
\usepackage{amsthm}

\newcommand{\E}{\mathbb{E}}
\newcommand{\Pb}{\mathbb{P}}

\newcommand{\Z}{\mathbb{Z}}

\theoremstyle{plain}
\newtheorem{lemma}{Lemma}

\newtheorem{cor}[lemma]{Corollary}
\newtheorem{obs}[lemma]{Observation}
\newtheorem{theo}[lemma]{Theorem}

\theoremstyle{definition}

\newtheorem*{claim*}{Claim}
\theoremstyle{remark}

\newtheorem{question}{Question}

\begin{document}

\tolerance 5000
\parskip 4pt
\parindent 0pt

\title[Escape of resources in distributed clustering processes]
{Escape of resources \\ in distributed clustering processes } 

\keywords{clustering proces, random spanning tree} \subjclass[2000]{60K35,
68M14}

\author[van den Berg]{J.\ van den Berg}
\address[J. van den Berg]{CWI and VU University Amsterdam}
\email{J.van.den.Berg@cwi.nl}

\author[Hil\'ario]{M.\ R.\ Hil\'ario}
\address[M.\ R.\ Hil\'ario]{IMPA, Rio de Janeiro}
\email{marcelo@impa.br}

\author[Holroyd]{Alexander E.\ Holroyd}
\address[A.\ E.\ Holroyd]{Microsoft Research and University of British Columbia}
\email{holroyd at math.ubc.ca}

\thanks{Funded in part by CNPq grants 140532/2007-2 (MRH) and NSERC (AEH)}


\date{23 July 2010}

\begin{abstract}
In a distributed clustering algorithm introduced by Coffman, Courtois,
Gilbert and Piret \cite{coffman91}, each vertex of $\mathbb{Z}^d$ receives an
initial amount of a resource, and, at each iteration, transfers all of its
resource to the neighboring vertex which currently holds the maximum amount
of resource.  In \cite{hlrnss} it was shown that, if the distribution of the
initial quantities of resource is invariant under lattice translations, then
the flow of resource at each vertex eventually stops almost surely, thus
solving a problem posed in \cite{berg91}.  In this article we prove the
existence of translation-invariant initial distributions for which resources
nevertheless escape to infinity, in the sense that the the final amount of
resource at a given vertex is strictly smaller in expectation than the
initial amount. This answers a question posed in \cite{hlrnss}.
\end{abstract}

\maketitle

\section{Introduction}
\label{sec:intro}

\subsection{Definitions and statement of the main result}

Consider, for $d \geq 1$, the $d$-dimensional integer lattice. This is the graph with vertex
set $\Z^d$, and edge set comprising all pairs of vertices $(x,y)$ ($=(y,x)$) with $|x-y| = 1$.
Here $| \cdot |$ denotes
the $1$-norm. We use the notation $\Z^d$ for this graph as well as for its vertex set. It will be clear from the context which of the two is meant.

The following model for `\textit{distributed clustering}' was introduced by Coffman, Courtois, Gilbert and Piret \cite{coffman91}.
To each vertex $x$ of the lattice $\Z^d$, we assign a random
nonnegative number $C_0(x) \in[0,\infty]$
which we regard as the initial amount of a `resource'
placed at $x$ at time $0$.  (The family $( C_0(x); x \in \Z^d )$ is not necessarily assumed independent).
Then we define a discrete-time evolution in which,
at each step, each vertex transfers its resource to the `richest' neighbouring vertex.

More precisely, the evolution is defined recursively
 as follows. Suppose that, at time $n$, the \textit{amount of
 resource} at each vertex $x$ is $C_n(x)$.
 Let $N(x) = \{y \in \mathbb{Z}^d \, : \, |x-y| \leq 1 \}$
 be the \textit{neighbourhood} of $x$ (note that it includes $x$ itself) and define
 $M_n (x) = \left\{y \in N(x) \, : \, ~ C_n(y) = \max_{z \in N(x)} C_n(z)\right\}$.
 Now let $v_n(x)$ be a vertex chosen uniformly at random in $M_n (x)$,
 independently for each $x$, and take:
 \begin{equation*}
 a_n(x) = \left\{
 \begin{array}{ll}
 x, & \text{if } C_n(x) = 0 \\
 v_n(x), & \text{if } C_n(x) > 0.
 \end{array}
 \right.
 \end{equation*}
 Finally, define
$$C_{n+1}(x) := \sum_{y \, : \, a_n(y) =x} C_n(y).$$

 For a fixed vertex $x$, the random variable $C_0(x)$ will be called the
 \textit{initial amount of resource} at $x$, and the family
 $\left( C_0(x); x \in \Z^d \right)$ will be called the \textit{initial configuration}.
 Analogously, $\left( C_n(x); x \in \Z^d \right)$ will be called the
 \textit{configuration at time} $n$. Note that $a_n(x)$ is the vertex to which the resources
 located at $x$ at time $n$ (if any) will be transferred during the $(n+1)$-th step
 of the evolution. We say that there is a \textit{tie} in $x$ at time $n$ if
 $C_n(x) > 0$ and the cardinality of $M_n(x)$ is strictly greater than one.
 In case this occurs, $a_n(x)$ is chosen uniformly at random among the
 vertices around $x$ that maximize $C_n$.
 Note that, apart from those possible tie breaks, all
 the randomness is contained in the initial configuration.
 As soon as a vertex has zero resource, its resource remains zero forever.
Also note that, when two or more vertices
 transfer their resources to the same vertex, these resources are
 added up. Thus this algorithm models a clustering process
 in the lattice starting from a disordered initial configuration.

 For a fixed vertex $x$, we use the notation $C_{\infty}(x)$ for $\lim_{n \to \infty} C_n(x)$
 in case this limit exists.
We write $\mathbb{E}$
 for expectation with respect to the underlying probability measure.

Our main result is the following theorem. The proof is given in
Section~\ref{sec:esca_mass}.
 \begin{theo}
 \label{theo:non_cons}
 Let $d \ge 2$. There exists a translation-invariant distribution for
 the initial configuration $\left(C_0(x); x \in \Z^d\right)$ such that, for each
 $x \in \Z^d$,
 \begin{equation} \label{eq-main}
 \mathbb{E}\left[C_\infty (x)\right] < \mathbb{E}\left[C_0(x) \right].
 \end{equation}
 \end{theo}

\subsection{Background and motivation}
Here is some more terminology.  If, for all sufficiently large $n$, we have that
$a_n(x) = x$ and $a_n(y) \neq x$ for all neighbours $y$ of $x$,
then we say that
the flow at $x$ \textit{terminates} after finitely many steps.
 In that case, the limit $C_{\infty}(x)$ is attained
 after finitely many iterations and will be called the \textit{final amount
 of resource at} $x$.
 If for all sufficiently large $n$ we have $a_{n+1}(x) = a_n(x)$, then we say that
$x$ eventually transfers its resource to the same fixed vertex.

The following stability questions for this process (formulated here similarly as in \cite{hlrnss})
have been investigated in the literature:

 \begin{question}
 \label{item3stability}
 Does each vertex eventually transfer its resource to the same fixed vertex almost
 surely?
 \end{question}
 \begin{question}
 \label{item1fixation}
 Does the flow at each vertex terminate after finitely many steps almost surely?
 \end{question}
 \begin{question}
 \label{item2conservation}
 If the answer to the previous question is affirmative, is the expected
 final amount of resource of a vertex equal to the expected initial
amount?
 \end{question}

 Of course the answers to the above questions may depend on the assumptions
 made about the distribution of the initial configuration.
 Note that if the answer to Question~\ref{item1fixation} is affirmative, then so is the answer
 to Question~\ref{item3stability}.  In that case, answering
 Question~\ref{item2conservation} is equivalent to answering the question whether
 the resource quantity that started on a given
 vertex will eventually stop moving almost surely. So, informally,
 Question~\ref{item1fixation}
 is related to fixation while Question~\ref{item2conservation} is related to
 conservation.

 Van den Berg and Meester \cite{berg91} considered the case $d=2$ and
 i.i.d.\ initial resource quantities. Using translation-invariance and
 symmetries of the system they proved that
 the answer  to Question~\ref{item3stability} is positive in the case
 that the initial quantities of resource have a continuous
 distribution.
 They also showed that, if the resources are integer valued,
 then Question~\ref{item1fixation} has a positive answer as well.

 Later, van den Berg and Ermakov \cite{berg98} considered again
 i.i.d.\ continuously distributed initial quantities of resource
 on the two-dimensional lattice. Using a percolation approach,
 they were able to relate Questions~\ref{item1fixation} and \ref{item2conservation}
 to a finite (but large) computation. By using
 Monte Carlo simulation, they obtained
 overwhelming evidence that the answer to these questions is
 positive for this case.

 In \cite{hlrnss} it was proved that, for every dimension and every translation-invariant
 distribution of the initial configuration, the answer to Question~\ref{item1fixation} is positive.
However, Question~\ref{item2conservation} was left open.
Our Theorem \ref{theo:non_cons} says that, for some initial distributions in this class,
 the answer to that question is negative.

 The conclusion of Theorem \ref{theo:non_cons} is false for
 $d=1$. To see that, suppose that the probability that
 the resource starting at the origin does not stop after
 finitely many steps is positive. Then, by translation invariance,
 there is, with positive probability, a positive density
 of vertices for which the initial resource will not stop
 after finite time. This implies that, with positive probability,
 there are infinitely many steps at which resource enters or leaves the origin,
contradicting the fixation result of \cite{hlrnss} mentioned in the previous paragraph.
 This argument can be generalized for example to any graph of the form
 $\Z \times G$, where $G$ is a finite vertex-transitive graph.
(For such graphs translation-invariance is replaced with automorphism-invariance).

In order to prove Theorem \ref{theo:non_cons}, we will
 construct a random collection (forest) of one-ended trees,
which is  embedded in $\Z^d$, in a translation-invariant way, and
 then assign resource quantities to the vertices
 in such a way that, during the evolution, each resource
 follows the unique infinite self-avoiding path to infinity in the forest.
 In Section \ref{sec:unif_span_tree} we present a short
 discussion about the existence of certain random
 forests on $\Z^d$.
 In Section \ref{sec:esca_mass}, Theorem \ref{theo:non_cons}
 is proved.
 In Section \ref{sec:conc_rema} we present some concluding
 remarks and open questions.

 \section{Translation-invariant forests on $\Z^d$}
 \label{sec:unif_span_tree}

 Let $G$ be an infinite graph. A \textit{forest} of $G$ is a subgraph of
$G$ that
 has no cycles. A \textit{tree} is a connected forest.
 A subgraph \textit{spans} $G$ if it contains
 every vertex of $G$.
 A \textit{spanning forest} (respectively \textit{tree}) on $G$ is a subgraph of
 $G$ that is a forest (respectively a tree) and that spans $G$.
 The \textit{leaves} of a forest $T$ are the vertices of $T$ that only have one neighbor
 in the forest. The \textit{number of ends} of a tree is the number of distinct self-avoiding
 infinite paths starting from a given vertex.
 A tree is said to be \textit{one-ended} if it has one end.

 We choose the $d$-dimensional
 integer lattice as the underlying graph.
 For this choice, the literature provides several
 constructions of random spanning forests with translation-invariant distributions, for example, the uniform
 spanning tree \cite{pemantle91}, and the minimal spanning tree
 \cite{alexander95}.
 To be explicit, we briefly discuss one construction,
 based on the two-dimensional minimal spanning tree.

 Let $E$ be the set of edges of the lattice $\Z^2$, and let
 $\left\{U_e; e \in E \right\}$ be a family of independent random
 variables distributed uniformly in the interval $[0,1]$.
 For each cycle of the lattice, delete the edge having
 the maximum $U$-value on the cycle.
 The resulting random graph is called \textit{(free) minimal spanning forest}
 and is known to be almost surely a one-ended tree
 which is invariant and ergodic under lattice translations (see \cite{lyons06}).

 For $d>2$, we can use the two-dimensional minimal
 spanning forest to construct a random
 forest in $\Z^d$ of which the distribution is invariant under lattice translations, and of which
every component is one-ended.
 We regard $\Z^d$ as $\Z^{2} \times \Z^{d-2}$ and in
 each `\textit{layer}' $\Z^2 \times \{z\}$ (where $z$ runs over $\Z^{d-2}$)
 we embed an independent copy of
 the two-dimensional minimal spanning tree $T_z$.
 The resulting subgraph of $\Z^d$ is a translation-invariant
 random spanning forest with one-ended components.
 This gives the following lemma.

 \begin{lemma}
 \label{lemma:span_tree}
 For each $d \geq 2$ there exists a translation-invariant random spanning
 forest on $\Z^d$, of which each connected component
is one-ended almost surely.
 \end{lemma}

 \begin{cor}
 \label{cor:span_tree}
 For each $d \geq 2$ there exists a translation-invariant random forest $T$
 on $\Z^d$, for which the following two properties hold almost surely. \\
(i) Every connected component of $T$ is one-ended. \\
(ii) Every edge of $\Z^2$ of which both endpoints are in $T$ is an edge of $T$.
 \end{cor}
 \begin{proof}
 Let $H$ be a spanning forest as in Lemma \ref{lemma:span_tree}
 and write $F$ for the set of its edges, and $V$ for the set of its
vertices. Let $\widetilde H$ be the forest with
 vertex set $\{2x \, : \, x \in V \} \cup \left\{ x+y \, : \, (x,y) \in F \right\}$ and
 edge set
 $E = \left\{ (2x, x+y)\, : \, (x,y)\in F \right\}$.
 Informally, $\widetilde H$ corresponds
 to the forest which is obtained when $H$ is scaled up by a factor $2$. Thus to
 each edge $(x,y)$ of $H$, there correspond two edges, $(2x, x+y)$ and $(x+y, 2y)$, in
 $\widetilde H$.
 Note that $\widetilde H$ is a random forest which is invariant under translations
of $2\Z^d$, and which has the property that every pair $x$, $y$ of its
vertices satisfying $|x - y| = 1$ is connected by an edge of $\widetilde H$.
To restore invariance under {\em all} translations of $\Z^d,$ let $W$ be a
uniformly random element of the discrete cube $\{0,1\}^d$, independent of
$\widetilde H$, and
 set $T = {\widetilde H} + W$.
 \end{proof}

 \section{Proof of main result}
 \label{sec:esca_mass}

In this section we fix $d \geq 2$.
 We will prove Theorem \ref{theo:non_cons} by giving an explicit
 construction of an initial configuration
 $(C_0(x), x \in \Z^d)$ whose distribution is
 translation-invariant and for which \eqref{eq-main} holds.

 Let $T$ be a random forest on $\Z^d$ as given by
 Corollary \ref{cor:span_tree}.
 For vertices $x$ and $y$ of $T$ we write $x \sim y$ if $(x,y)$ is an edge of $T$. We define a (random) partial order $\leq$ on $\Z^d$ by setting $y \leq x$ if and only if $x$ and $y$ are vertices of $T$ and
 $x$ belongs to the unique infinite self-avoiding path in $T$ starting at $y$.
 If $y \leq x$ we say that $x$ is an
 \textit{ancestor} of $y$ and that $y$ is a \textit{descendant} of $x$.
 If $y \leq x$ and $x \sim y$ we say that $x$ is a \textit{parent}
 of $y$. Note that every vertex of $T$ has a unique parent. Moreover, for every vertex $x$ of $T$, exactly one vertex
in $\{y \, : \, y \sim x\}$ is the parent of $x$, and the others are
descendents of $x$.

 We now define, for each $x \in \Z^d$, the initial quantity of resource
 at $x$ by:
 \begin{equation}
 \label{eq:conf_init}
 C_0 (x) = 
  \begin{cases}
   \sum_{y\in\Z^d} \mathbf{1}[y\leq x],   &\text{if } x \in T \\
   0, & \text{otherwise}.
  \end{cases} 
 \end{equation}

 Note that, if $x \in T$, then $C_0 (x)$ is the number of descendants of $x$.
 Since every connected component of $T$ is one-ended, it follows from the definitions that this number is finite.
 Also note that, since the distribution of $T$ is invariant under the translations
 of $\Z^d$, so is that of the family $\left( C_0(x); x \in \Z^d \right)$.

 We now define a nested (`decreasing') sequence of forests that will be shown to describe
 the dynamics of resources when $C_0(x)$ is given by (\ref{eq:conf_init}).
For a forest $S$, let $\phi(S)$ denote the forest obtained from
 $S$ by deleting all its leaves.
 Let $T_0 = T$ and, for $n = 1, 2, \ldots $, define inductively
 $T_n = \phi(T_{n-1})$.

The following observation follows easily from the definitions.

\begin{obs}\label{obs-Tn}
Let $y$ be a vertex of $T$, let $x$ be the parent of $y$, and let $n \geq 0$.
Then $x$ is in $T_{n+1}$ if and only
if $y$ is in $T_n$.
\end{obs}

\begin{lemma} \label{obs-end}
 For every vertex $x$ in $T$,
 there is a finite index $n_0$ (depending on $x$) such that, for all $n \geq n_0$,
 $x$ does not belong to $T_n$.
\end{lemma}
\begin{proof}
By Observation \ref{obs-Tn}, $n_0(x)$ is at most $1$ plus the number of descendents of $x$.
As we mentioned before, this number is finite.
\end{proof}
 \begin{lemma}
 \label{lemm:dyna}
 Suppose that, for all $x$, $C_0 (x)$ is given by (\ref{eq:conf_init}).
 Then for all $n \geq 0$,
 \begin{equation}
 \label{eq:flux_stab}
 C_n (x) 
  \begin{cases}
  > \sum_{y: y \sim x, \, y \leq x} C_n(y),
   & \text{\rm if } x \in  T_n; \\
  = 0, & \text{\rm if } x \notin T_n.
  \end{cases} 
 \end{equation}
 \end{lemma}

 \begin{proof} We use induction on $n$.
 To verify (\ref{eq:flux_stab}) for $n=0$ we note that, if $x$ belongs to $T_0 (= T)$ then,
 by (\ref{eq:conf_init}),
 \begin{equation*}
 \sum_{\substack{y: y \sim x, \\ y \leq x}} C_0(y) =
 \sum_{\substack{y:y \sim x, \\ y \leq x}} \; \sum_{z \in \Z^d} \mathbf{1}[z \leq y] =
 \sum_{z \in \Z^d \setminus \{x\}} \mathbf{1}[z \leq x] = C_0(x) - 1.
 \end{equation*}

 Now, suppose that (\ref{eq:flux_stab}) holds for a given $n$.
 Since $T$ was taken as in Corollary \ref{cor:span_tree},
 two vertices of $T_n$ which are
 adjacent in $\Z^d$ must be linked by an edge of $T_n$.
By this and \eqref{eq:flux_stab} it follows that, for each vertex $z$ of $T_n$, $a_n(z)$ is
the parent of $z$. Therefore, and because $C_n \equiv 0$ outside $T_n$, we have
\begin{align}
C_{n+1}(x) &= 0  &\text{ if } x \notin T_{n+1}; \label{eq-a1}\\
C_{n+1}(x) &= \sum_{y: y \sim x, \, y \leq x} C_n(y)  &\text{ if } x \in T_{n+1}. \label{eq-a2}
\end{align}

%

By applying \eqref{eq-a2}, \eqref{eq:flux_stab} and Observation \ref{obs-Tn} (and noting that \eqref{eq-a2} also holds
for $x \in T_n \setminus T_{n+1}$, since then both sides of \eqref{eq-a2} are equal to $0$),
we get, for $x \in T_{n+1}$,

\begin{equation} \label{eq-a3}
C_{n+1}(x) = \sum_{\substack{y:y \sim x,\\ y \leq x}} C_n(y)
> \sum_{\substack{y:y \sim x,\\ y \leq x}} \,\,
\sum_{\substack{z:z \sim y,\\ z \leq y}} C_n(z)
= \sum_{\substack{y:y \sim x,\\ y \leq x}} C_{n+1}(y).
\end{equation}
Now \eqref{eq-a1} and \eqref{eq-a3} complete the induction step, and the proof of Lemma \ref{lemm:dyna}.
\end{proof}

\begin{proof}[Proof of Theorem \ref{theo:non_cons}]  Let the initial configuration be defined as in \eqref{eq:conf_init}.  Let $x \in \Z^d$. By Lemma \ref{lemm:dyna} and Lemma \ref{obs-end},
we have that almost surely $C_n(x) = 0$ for all sufficiently large $n$.
Hence $C_{\infty}(x) = 0$ almost surely. On the other hand, it is clear that $C_0(x) > 0$ with positive probability, and
hence $\E[C_0(x)] > 0$.
\end{proof}

 \section{Concluding remarks and open problems}
 \label{sec:conc_rema}
At the end of the proof of Theorem \ref{theo:non_cons}
we mentioned the obvious fact that $\E[C_0(x)] > 0$ for every $x$. It turns out that this
expectation is even $\infty$. Indeed, we have
\begin{align*}
 \E[C_0(x)] & = \sum_{y \in \Z^2} \Pb[y \leq x]
              = \sum_{y \in \Z^2} \Pb[x+y \leq x] \\
                  & = \sum_{y \in \Z^2} \Pb[x \leq x-y]
                    = \sum_{y \in \Z^2} \Pb[x \leq y ] = \infty,
 \end{align*}
 where the second and forth equality follow
 by relabeling, the third equality follows by
 translation-invariance and
 the last inequality follows from the fact
 that $x$ has infinitely many ancestors
 almost surely.

 We have not been able to construct an example where
 resources escape to infinity but the
 initial amount of resource at a given vertex has finite expectation.
 It is an interesting question whether such
 examples exist.

 In particular, in our construction,
 the initial configuration was chosen in such a way that, almost surely, the induced
 dynamics takes place in a forest with one-ended components, embedded in $\Z^d$ and,
 at each step, the resources are transferred from every vertex with non-zero
 resource to its parent. It is not clear if for every initial configuration with these
properties
 the expectation of the initial amount of resource of a vertex is infinite.
We state these considerations more formally by the following two questions.

 \begin{question}
 \label{item4conservation}
 Suppose that $(C_0(x); x \in \Z^d)$ has a translation-invariant distribution and
 is positive exactly on the vertices of a forest with one-ended components. Furthermore, suppose that
 during the $n$-th step of the evolution, every vertex for which $C_{n-1}(x) > 0$ transfers
 its resource to its parent.
 Is it the case that $\mathbb{E}[C_0(x)] = \infty?$
 \end{question}

 \begin{question}
 \label{item5conservation}
 Does there exist a translation-invariant distribution for the initial configuration for which
 $\E[C_{\infty}(x)] < \E[C_0(x)] < \infty$?
 \end{question}

 A negative answer to Question \ref{item4conservation} would
 yield a positive answer to Question \ref{item5conservation}.

\textbf{Acknowledgments}
 Marcelo Hil\'ario thanks CWI (Amsterdam), Eurandom (Eindhoven)
 and Microsoft Research (Redmond)
 for their hospitality.
 The authors thank Leonardo T. Rolla, Scott Sheffield and Vladas Sidoravicius
 for important comments and suggestions.

\end{document}